\documentclass[letterpaper, 10 pt, conference]{ieeeconf}

\IEEEoverridecommandlockouts

\usepackage{balance}

\usepackage[english]{babel}
\usepackage{amssymb,amsmath} 
\usepackage[latin1,utf8]{inputenc}
\usepackage{graphicx}
\usepackage{epstopdf}
\usepackage{subfigure}
\usepackage[]{algorithm}
\usepackage{algorithmic}

\usepackage[noadjust]{cite}
\usepackage{soul}
\usepackage[utf8]{inputenc}
\usepackage[mathscr]{euscript}
\usepackage{hyperref}

\usepackage[dvipsnames]{xcolor}

\usepackage{bm}

\usepackage{sidecap}

\usepackage[normalem]{ulem}

\usepackage{tikz}
\usetikzlibrary{fit,positioning,arrows,automata}

\def\minwrt[#1]{\underset{#1}{\text{minimize }}}
\def\argminwrt[#1]{\underset{#1}{\text{arg min }}}
\def\maxwrt[#1]{\underset{#1}{\text{maximize }}}
\def\argmaxwrt[#1]{\underset{#1}{\text{arg max }}}
\def\maxemphwrt[#1]{\underset{#1}{\text{\emph{maximize} }}}

\def\bK{{\bf K}}
\def\bM{{\bf M}}

\def\bU{{\bf U}}

\def\bC{{\bf C}}

\def\ccT{{\mathcal{T}}}

\def\ccP{{\mathcal{P}}}

\def\RR{{\mathbb{R}}}

\newcommand{\diag}{\text{diag}}

\newcommand{\SpeciesMtx}{\mathfrak{R}}
\newcommand{\SpeciesMtxElem}{\mu}

\newcommand{\newC}{\bC}

\graphicspath{{./Figures/}}

\newtheorem{theorem}{Theorem}
\newtheorem{remark}{Remark}

\title{\LARGE \bf
Mean field type control with species dependent dynamics\\ via structured tensor optimization
}

\author{Axel Ringh, Isabel Haasler, Yongxin Chen, Johan Karlsson
\thanks{This work was supported by the Wallenberg AI, Autonomous Systems and Software Program (WASP) funded by the Knut and Alice Wallenberg Foundation, the Knut and Alice Wallenberg Foundation under grant KAW 2021.0274, the NSF under grant 1942523 and 2206576,
the Swedish Research Council (VR) under grant 2020-03454, and KTH Digital Futures.}
\thanks{A.~Ringh is with Department of Mathematical Sciences, Chalmers University of Technology and the University of Gothenburg, SE-412 96 Gothenburg, Sweden. {\tt\small axelri@chalmers.se}}%
\thanks{I.~Haasler is with Signal Processing Laboratory, LTS 4, \'{E}cole Polytechnique F\'{e}d\'{e}rale de Lausanne, Lausanne, Switzerland. {\tt\small isabel.haasler@epfl.ch}}%
\thanks{Y.~Chen is with the School of Aerospace Engineering, Georgia Institute of Technology, Atlanta, GA, USA. {\tt\small yongchen@gatech.edu}}\\%
\thanks{J.~Karlsson is with the Division of Optimization and Systems Theory, Department of Mathematics, KTH Royal Institute of Technology, Stockholm, Sweden. {\tt\small johan.karlsson@math.kth.se}}%
}

\begin{document}

\maketitle
\thispagestyle{empty}
\pagestyle{empty}

\renewcommand{\algorithmicrequire}{\textbf{Input:}}
\renewcommand{\algorithmicensure}{\textbf{Output:}}
\algsetup{indent=5pt}

\begin{abstract}
In this work we consider
mean field type control problems
with multiple species that have different dynamics.
We formulate the discretized problem using a new type of entropy-regularized multimarginal optimal transport problems where the cost is a decomposable structured tensor. 
A novel algorithm for solving such problems is derived, using this structure and leveraging recent results in 
entropy-regularized optimal transport.
The algorithm is then demonstrated on a numerical example in robot coordination problem for search and rescue, where three different types of robots are used to cover a given area at minimal cost.
\end{abstract}

\section{Introduction}\label{sec:introduction}

In recent years, mean field type control problems 
have emerged as a powerful tool for analysis and control of large-scale dynamical systems consisting of subsystem that are also dynamical systems.
It provides a framework for modeling the behavior of a large population of interacting agents, where
i) each individual's decision is negligible to others at the individual level, but where the actions are significant when aggregated, and ii) the number of agents is too large to model each one individually. In such cases, one instead often seek
a model that explains the aggregate behavior of the population  \cite{huang2006large, lasry2007mean, huang2012social, djehiche2017mean}.
Mean field type control problems are density optimal control problems where the density abides to a controlled Fokker-Planck equation with distributed control \cite{lasry2007mean, cardaliaguet2015second, bensoussan2013mean, chen2018steering}. For example, potential mean field games are a particular type of such models \cite{lasry2007mean, BenCarDiNen19}. 

In basic formulations of mean field type control problems, all agents are equivalent in the sense that they all have the same dynamics and they all have the same objective function which they try to minimize. However, an important generalization
is the multispecies setting, where the population consists of several different types of agents \cite{huang2006large, lasry2007mean, lachapelle2011mean, achdou2017mean, cirant2015multi, bensoussan2018mean}.
This type of problem occurs in, e.g., coordination of multiple types of robots, where the robots have different properties (such as movement speed, movement capabilities, cost, etc.), but they still have a common goal of achieving a given task as efficiently as possible.

Recently optimal transport has been successfully used to address a number of problems in control, see, e.g., \cite{bonnet2021necessary, terpin2023dynamic, zorzi2020optimal}.
In the seminal paper \cite{benamou2000computational}, certain optimal transport problems were formulated as density control problems over the continuity equation, and this idea can be generalized to allow for optimal transport problems that have general underlying dynamics \cite{hindawi2011mass, chen2017optimal}. Moreover, the recently developed Sinkhorn method,  for numerically solving large-scale optimal transport problems \cite{cuturi2013sinkhorn, peyre2019computational}, is closely related to the density control formulation in \cite{benamou2000computational}. In fact, the added entropy regularization leading to the Sinkhorn method corresponds to adding a stochastic term to the underlying particle dynamics, which leads to a controlled Fokker–Planck equation in the density control problem \cite{chen2016relation, chen2020stochastic}.
This has been used to develop methods for solving potential mean field games \cite{BenCarDiNen19, ringh2021efficient, ringh2021graph}, 
by formulating the potential mean field game as a mean field type control problem and then formulated the latter as a multimarginal optimal transport problem. Due to the Markov property, such multimarginal optimal transport problems have a graph-structured cost, and this
type of graph structures has also been used to develop efficient computational methods to solve problems in control \cite{haasler2020optimal, haasler2021multimarginal}, estimation \cite{haasler19ensemble}, and information fusion \cite{elvander2020multi}.

In this paper, we consider mean field type control problems with multiple species that have different dynamics. First, we reformulate the problem as a multimarginal optimal transport problem. However, in this case the resulting problem turns out to have a cost function that is no longer graph-structured; instead, the cost function is a decomposable structured tensor, i.e., a multi-indexed matrix, which can be represented by a hypergraph. Next, for this type of structured multimarginal optimal transport problem, we develop an efficient solution algorithm. Finally, we illustrate the developed method on an example in coordination of multiple types of robots in a search-and-rescue type mission.

The outline of the paper is as follows: in Section \ref{sec:background} we briefly review the areas of multimarginal optimal transport, and mean field control problems. In Section \ref{sec:main} we formulate the multispecies mean field control problem as a structured multimarginal optimal transport problems, discretize it,  and present an efficient numerical method for computing the optimal solution of the latter. In Section \ref{sec:example} we present a detailed numerical example of robot coordination, and finally in Section \ref{sec:conclusions} we present conclusions and future directions.

\section{Background}\label{sec:background}

\subsection{Multimarginal optimal transport} \label{subsec:ot}

The optimal transport problem is a classic problem in mathematics that involves finding the most efficient way of moving mass to transform one distribution into another \cite{villani2003topics}.
The multimarginal optimal transport problem is an extension of this concept that deals with multiple distributions \cite{ruschendorf1995optimal, gangbo1998optimal, pass2015multi, benamou2015bregman}.
Here, we focus on the discrete case, where the marginal distributions are represented by a finite set of nonnegative vectors\footnote{To simplify the notation, we assume that all the marginals have the same number of elements, i.e., $\mu_j\in \RR^N$. This can easily be relaxed.} $\mu_1, \ldots, \mu_\ccT \in \RR^N_+$. The transport plan and cost are both represented by $\ccT$-mode tensors, $\bM \in{\RR^{N^\ccT}}$ and $\bC \in{\RR^{N^\ccT}}$, respectively, and the marginal distributions of the transport plan are given by projections $P_j(\bM) \in \RR^N_+$, where%
\footnote{For notational convenience, we will in the remainder of the text write this type of sum as $\sum_{i_1, \ldots, i_{j-1}, i_{j+1},\ldots, i_\ccT}\bM_{i_1 \ldots i_\ccT}.$}
\begin{equation*}
(P_j(\bM))_{i_j} := 
\sum_{i_1=1}^N \cdots \sum_{i_{j-1}=1}^N
\sum_{i_{j+1}=1}^N \cdots \sum_{i_{\ccT}=1}^N   \bM_{i_1 \ldots i_\ccT}.
\end{equation*}

The discrete multimarginal optimal transport problem can then be formulated as
\begin{subequations}\label{eq:discr_multimarginal}
\begin{align}
\minwrt[\bM \in \RR^{N^{\ccT}}_+] & \quad \langle \bC, \bM \rangle \\
\text{subject to } & \quad P_j(\bM) = \mu_j,\; j \in \Gamma,
\end{align}
\end{subequations}
where $\langle\bC, \bM\rangle := \sum_{i_1, \ldots, i_\ccT} \bC_{i_1 \ldots i_\ccT} \bM_{i_1 \ldots i_\ccT}$ is the standard inner product, and where 
 we impose constraints on the marginals corresponding to the index set $\Gamma \subset \{1, \ldots, \ccT\}$.
Although the optimal transport problem \eqref{eq:discr_multimarginal} is a linear program, it can be challenging to solve it numerically due to the large number of variables. A popular method for approximately solving \eqref{eq:discr_multimarginal} is to perturb the problem by adding a small $\epsilon > 0$ times the entropy term
\[
D(\bM) := \sum_{i_1, \ldots, i_\ccT} \big( \bM_{i_1 \ldots i_\ccT}\log(\bM_{i_1 \ldots i_\ccT}) - \bM_{i_1 \ldots i_\ccT} + 1 \big)
\]   
to the cost function and use Sinkhorn iterations to solve the resulting problem \cite{cuturi2013sinkhorn, peyre2019computational}.
The optimal transport plan for the perturbed problem is
of the form
$\bM = \bK \odot \bU$,
where $\bK = \exp(- \bC/ \epsilon)$, and 
$\bU_{i_1 \dots i_\ccT} = \prod_{j \in \Gamma} (u_j)_{i_j}$
a rank-one tensor,%
\footnote{The notation $(u_j)_{i_j}$ means the $i_j$th element of the vector $u_j$; we will use analogous notation for tensors in general.}
see \cite{benamou2015bregman, elvander2020multi}.
Sinkhorn's method iteratively updates $u_j$ as
\begin{equation} \label{eq:sinkhorn_multi}
u_j \leftarrow u_j \odot \mu_j \oslash P_j(\bK \odot \bU), \quad \text{ for } j \in \Gamma,
\end{equation}
where $ \odot$ and $\oslash$ means pointwise multiplication and pointwise division, respectively,
and the algorithm converges (linearly) to an optimal solution of the perturbed problem \cite{luo1992convergence,tseng1990dual}.
In the multimarginal case, computing $P_j(\bK \odot \bU)$ suffers from the curse of dimensionality.
However, in some cases, structures in the underlying cost $\bC$ can be used to circumvent these issues, for instance when the cost decouples into pairwise interactions according to a graph-structure \cite{benamou2015bregman, haasler19ensemble, haasler2020optimal, elvander2020multi, haasler2021multimarginal, haasler2021scalable, haasler2020multi, singh2022inference, altschuler2020polynomial, ringh2021efficient, ringh2021graph, fan2022complexity}.

\subsection{Multispecies mean field control problems}

Consider a set of infinitesimal agents moving in a state space $X\subset \RR^n$. Assume that they belong to $L$ different classes, and that each infinitesimal agent of species $\ell \in \{ 1, \ldots, L \}$ obeys the dynamics
\begin{equation} 
dx_\ell(t) = f_\ell(x_\ell)dt + B_\ell(x_\ell)\big(v_\ell dt + \sqrt{\epsilon}dw_\ell\big), \label{eq:SDE_species}
\end{equation}
subject to the initial condition $x_\ell(0)=x_{0, \ell}$, where the latter is a realization from a distribution $\rho_{0,\ell}(x)$.
Moreover, $w_\ell$, for $\ell \in \{ 1, \ldots, L \}$, are $n$-dimensional Wiener processes that are independent of each other.
We also assume that $f_\ell:X\to \RR^n$ and $B_\ell:X\to \RR^{n\times n} $ are continuously differentiable with bounded derivatives. Then, under suitable conditions on the (Markovian) feedback $v_\ell$, there exists a unique solution almost surely
to \eqref{eq:SDE_species}, cf.~\cite[Thm.~V.4.1]{fleming1975deterministic}, \cite[pp.~7-8]{bensoussan2013mean}.
Moreover, the density of the particles of species $\ell$, $\rho_\ell$, is the solution of a controlled Fokker-Planck equation. Therefore, the multispecies mean field control problem
is defined as the density optimal control problem
\begin{subequations}\label{eq:densitycontrol_multispecies}
    \begin{align}
        \minwrt[\substack{\rho, \rho_\ell, v_\ell}] & \; \int_0^1 \!\!\! \int_X \sum_{\ell = 1}^L \frac{1}{2} \|v_\ell\|^2 \rho_\ell \, dx dt \nonumber \\
        & + \! \int_0^1 \!\! \mathcal{F}_t(\rho(t,\cdot)) dt  + \mathcal{G}(\rho(1,\cdot)) \label{eq:densitycontrol_multispecies1} \\
        & + \sum_{\ell=1}^L \left( \int_0^1 \mathcal{F}^\ell_t(\rho_\ell(t, \cdot)) dt + \mathcal{G}^\ell(\rho_\ell(1, \cdot)) \right) \nonumber \\        
        \text{subject to } & \;  \frac{\partial \rho_\ell}{\partial t} + \nabla\cdot((f_\ell + B_\ell v_\ell) \rho_\ell ) \label{eq:densitycontrol_multispecies2} \\
        & \quad - 
\frac{\epsilon}{2} \sum_{i,k=1}^n \frac{\partial^2 ((\sigma_\ell)_{ik}  \rho_\ell)}{\partial x_i \partial x_k}
= 0, \quad \ell = 1, \ldots L,  \nonumber \\
        & \; \rho_\ell(0,\cdot) = \rho_{0,\ell}, \quad \rho(t, \cdot) = \sum_{\ell=1}^L \rho_\ell(t, \cdot), \label{eq:densitycontrol_multispecies3}	
    \end{align}
\end{subequations}
cf.~\cite[Chp.~2 and 4]{bensoussan2013mean}.
Here $\nabla \cdot$ denotes the divergence operator, $\sigma_\ell := B_\ell B_\ell^T$,
and  $\mathcal{F}_t$, $\mathcal{G}$, $\mathcal{F}^\ell_t$ and $\mathcal{G}^\ell$
are functionals on on $L_2\cap L_\infty$. These functionals are the costs that the species are trying to minimize by their behavior: $\mathcal{F}^\ell_t$ and $\mathcal{G}^\ell$ are species-dependent costs, where the former is the running cost and the latter is the terminal cost. $\mathcal{F}_t$ and $\mathcal{G}$ are cooperative costs that link the species together by acting on the total density of all species. We assume that all these functionals
are proper, convex, and lower-semicontinuous.
Moreover, we assume that $\mathcal{F}_t$ and $\mathcal{F}^\ell_t$ are piece-wise continuous in time.

\section{Discretization and solution via tensor optimization}\label{sec:main}

Recently, an approach for solving some types of potential mean field games was
proposed in \cite{ringh2021efficient, ringh2021graph}. The approach is based on formulating the problem on path space, i.e., on $C([0,1], X) := $ the set of continuous functions from $[0,1]$ to $X$, and then discretizing the problem in time and space, which results in a tensor optimization problem.
Here we generalize this approach in order to derive a numerical solution algorithm for multispecies mean field control problems of the form \eqref{eq:densitycontrol_multispecies}.

\subsection{Discretization of the problem}
Let $\ccP^{v_\ell}$ denote the distribution of species $\ell$ on path space, induced by the controlled process \eqref{eq:SDE_species}.
Then $\ccP^{v_\ell}_{t} = \rho_\ell(t, \cdot)$, where $\ccP^{v_\ell}_{t}$ is the marginal of $\ccP^{v_\ell}$ corresponding to time $t$, and $\rho_\ell$ is the solution to \eqref{eq:densitycontrol_multispecies2} with initial condition $\rho_\ell(0,\cdot) = \rho_{0,\ell}$.
Moreover, let $\ccP^0_\ell$ denote the corresponding uncontrolled process ($v_\ell \equiv 0$) with initial density $\rho_{0,\ell}$. 
By the Girsanov theorem (see, e.g., \cite[pp.~156-157]{Follmer88}), we get that
\begin{equation}
  \frac{1}{2}\int_X \int_0^1 \!\! \|v_\ell\|^2 \rho_\ell dt dx
  = 
  \epsilon {\rm KL} (\ccP^{v_\ell} \| \ccP^0_\ell) \label{eq:relativeentropy}
\end{equation}
where ${\rm KL} (\cdot \| \cdot)$ is the Kullback-Leibler divergence, see, e.g., \cite{chen2016relation, chen2018steering, BenCarDiNen19,leonard2012schrodinger, leonard2013schrodinger}.
Utilizing \eqref{eq:relativeentropy}, problem \eqref{eq:densitycontrol_multispecies} can be reformulated as an optimization problem over path space measures that corresponds to a generalized entropy-regularized multimarginal optimal transport problem, see \cite{ringh2021efficient, ringh2021graph}.
In particular, discretizing the space into points $x_1,\dots, x_N$, and considering time steps $j\Delta t$, for $j=1,\dots,\ccT$, where $\Delta t = 1/\ccT$, the term \eqref{eq:relativeentropy} takes the form $\langle \bC_\ell, \bM_\ell \rangle + \epsilon D(\bM_\ell)$, where 
the tensor $\bM_\ell\in \RR^{N^{\ccT+1}}$ describes the flow of agents in class $\ell$, and the cost tensor $\bC_\ell\in \RR^{N^{\ccT+1}} $ describes the associated cost of moving agents. More precisely,
$(\bC_\ell)_{i_0 \ldots i_{\ccT}} = \sum_{j=0}^{\ccT-1} C_{\ell i_j i_{j+1}}$,
where $C_{\ell i k}$ 
is the optimal cost for moving a unit mass of species $\ell$ from point $x_i$ to $x_k$ in one time step, i.e.,
\begin{equation}\label{eq:C_elem}
    C_{\ell i k} \; = \;
    \begin{cases}
    \displaystyle \; \minwrt[{v\in L_2([0,\Delta t])}] & \; \displaystyle \int_{0}^{\Delta t} \frac{1}{2} \| v \|^2 dt \\
    \; \text{subject to } & \; \displaystyle \dot{x} = f_\ell(x) + B_\ell(x)v \\
    & \; \displaystyle x(0) = x_{i}, \quad x(\Delta t) = x_{k}.
    \end{cases}
\end{equation}
Thus, the discretization of \eqref{eq:densitycontrol_multispecies} takes the form
\begin{subequations}\label{eq:multispecies_discrete}
    \begin{align}
    \!\!\!\!\! \minwrt[\substack{\bM_\ell, \mu_j, \mu_j^{(\ell)}\\ j = 1,\ldots, \ccT \\ \ell = 1, \ldots, L}] & \, \sum_{\ell = 1}^L \big( \langle \bC_\ell, \bM_\ell \rangle + \epsilon D(\bM_\ell) \big) + \Delta t \sum_{j=1}^{\ccT-1} F_{j}(\mu_j)  \nonumber \\[-12pt]
    & + \! G(\mu_\ccT) \! + \! \sum_{\ell = 1}^L \!\! \left( \! \Delta t \! \sum_{j=1}^{\ccT-1} \! F_j^\ell ( \mu_j^{(\ell)}) \! + \! G^\ell(\mu_\ccT^{(\ell)}) \! \right) \label{eq:multispecies_discrete1}\\
    \!\!\!\!\! \text{subject to } & \, P_j(\bM_\ell) = \mu_j^{(\ell)},\, j=1, \ldots, \ccT, \, \ell = 1, \ldots, L, \label{eq:multispecies_discrete2} \\
    & \, P_0(\bM_\ell) = \mu_{0, \ell}, \; \ell = 1, \ldots, L, \label{eq:multispecies_discrete3} \\
    & \, \sum_{\ell = 1}^L \mu_j^{(\ell)} = \mu_j, \;\;\; j=0, \ldots, \ccT \label{eq:multispecies_discrete4}
    \end{align}
\end{subequations}
where $\mu_{0, \ell}$ are discrete approximations of $\rho_{0, \ell}$.

\subsection{Solution method based on tensor optimization}
Note that \eqref{eq:multispecies_discrete} consists of $L$ coupled tensor optimization problems as in \cite{ringh2021efficient}, which are coupled through the constraint \eqref{eq:multispecies_discrete4} and the cost imposed on $\mu_j$, for $j = 1,\ldots, \ccT$, in \eqref{eq:multispecies_discrete1}.
Next, we reformulate \eqref{eq:multispecies_discrete} into a single tensor optimization problem (cf. \cite{ringh2021graph, haasler2021scalable}) by ``stacking together'' the tensors $\bM_\ell$ for $\ell=1,\dots,L$ to form a $(\ccT+2)$-mode tensor $\bM \in \RR^{L\times N^{\ccT+1}}$, where the index $-1$ refers to the species. That is, its elements are given by $\bM_{\ell i_{0} \dots i_\ccT} = (\bM_\ell)_{i_{0} \dots i_\ccT}$, and $\bM_{\ell i_{0} \dots i_\ccT}$ is the amount of mass of species $\ell$ that moves along the path $x_{i_{0}}, \dots, x_{i_\ccT}$.
Therefore, the additional marginal $\mu_{-1} = (P_{-1}(\bM))_\ell \in \RR_+^L$ describes the total mass of the densities for the different species.
Moreover, the bi-marginal projection $P_{-1, j}(\bM)\in \RR^{L\times N}$, defined by
\[
(P_{-1, j}(\bM))_{\ell i_j} := \sum_{i_0,\dots,i_{j-1},i_{j+1},\dots,i_\ccT} \bM_{\ell i_{0} \dots i_\ccT}, 
\]
satisfies $P_{-1, j}(\bM) = [\mu_j^{(1)}, \ldots, \mu_j^{(L)}]^T$, and $P_j(\bM)$ is the total distribution $\mu_j$ at time $j\Delta t$.

Finally note that $\sum_{\ell = 1}^L D(\bM_\ell) = D(\bM)$, and hence problem \eqref{eq:multispecies_discrete} can be written as the tensor optimization problem
\begin{subequations}  \label{eq:multispecies_discrete_rewritten}
\begin{align}
    \minwrt[\substack{\bM, \mu_j, R^{(-1, j)}\\ j=1,\ldots, \ccT}] & \; \langle \newC, \bM \rangle + \epsilon D(\bM) + \Delta t \sum_{j=1}^{\ccT-1} F_{j}(\mu_j) + G(\mu_\ccT) \nonumber \\[-10pt]
    & \; + \sum_{j=1}^{\ccT-1} \mathscr{F}_j^{L}( R^{(-1, j)} ) + \mathscr{G}^{L}(R^{(-1, \ccT)}) \label{eq:multispecies_discrete_rewritten1}\\
    \text{subject to } & \; P_j(\bM) = \mu_j,\qquad\qquad \; j=1,\ldots, \ccT, \label{eq:multispecies_discrete_rewritten2} \\
    & \; P_{-1,j}(\bM) = R^{(-1, j)}, \quad j=1,\ldots, \ccT, \label{eq:multispecies_discrete_rewritten4} \\
    & \; P_{-1,0}(\bM) = \SpeciesMtx^{(-1, 0)} , \label{eq:multispecies_discrete_rewritten3}
\end{align}
\end{subequations}
where $\SpeciesMtx^{(-1,0)}=[\SpeciesMtxElem_{0, 1},\ldots, \SpeciesMtxElem_{0, L}]^T\in \RR_+^{L\times N}$ and where
$\mathscr{F}_j^{L}(R^{(-1, j)}) = \sum_{\ell = 1}^L \Delta t F_j^\ell( \mu_j^{(\ell)})$, $j = 1, \ldots, \ccT$,
and similarly for $\mathscr{G}^{L}$.
Here, the cost tensor $\newC$ is given by
$\newC_{\ell i_0 \ldots i_\ccT} = \sum_{j=0}^{\ccT-1} C_{\ell i_j i_{j+1}}$, which means that the problem has a structure as the hypergraph illustrated in Figure~\ref{fig:multispecies_triangles}.

 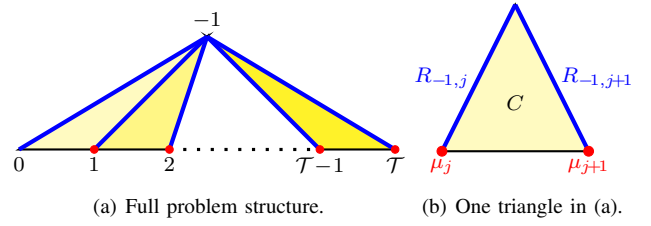
\begin{figure}
\subfigure[Full problem structure.]{
	\begin{tikzpicture}[scale=0.5]
	\footnotesize

 \coordinate (-1) at (5,3); 
\coordinate (0) at (0,0);
\coordinate (1) at (2,0);
\coordinate (2) at (4,0);
\coordinate (T-1) at (8,0);
\coordinate (T) at (10,0);

\draw[thick,fill=yellow!30] (-1) node[anchor=south]{$-1$} -- (0) node[anchor=north]{$0$} -- (1) node[anchor=north]{$1$} -- cycle;
\draw[thick,fill=yellow!60] (-1) -- (1) -- (2) node[anchor=north]{$2$} -- cycle;
\draw[thick,fill=yellow!90] (-1) -- (T-1) node[anchor=north]{$\ccT\!-\!1$} -- (T) node[anchor=north]{$\ccT$} -- cycle;
\draw[loosely dotted, very thick] (2) -- (T-1);

\draw[blue, ultra thick] (-1) -- (0);
\draw[blue, ultra  thick] (-1) -- (1);
\draw[blue, ultra thick] (-1) -- (2);
\draw[blue, ultra thick] (-1) -- (T-1);
\draw[blue, ultra thick] (-1) -- (T);

\draw[red, fill=red]   (1) circle (.1);
\draw[red, fill=red]   (2) circle (.1);
\draw[red, fill=red]   (T-1) circle (.1);
\draw[red, fill=red]   (T) circle (.1);
	\end{tikzpicture} }
	\hspace*{-15pt}
\subfigure[One triangle in (a).]{
	\begin{tikzpicture}[scale=0.65]
	\footnotesize
	\tikzstyle{main}=[circle, fill=white, minimum size = 12mm, thick, draw =black!80, node distance = 10mm]
	
\coordinate (-1) at (1.5, 3) ;
\coordinate (j) at (0,0);
\coordinate (j+1) at (3,0);	
	
\draw[thick,fill=yellow!30] (-1) -- (j)  -- (j+1) -- cycle;		
\draw[blue, ultra thick] (-1) --node[left]{$R_{-\!1,j}$} (j);
\draw[blue, ultra thick] (-1) --node[right]{$R_{-\!1,j\!+\!1\!}$} (j+1);

\draw[red, fill=red]   (j) circle (.1) node[below, red]{$\mu_{j}$}; 
\draw[red, fill=red]   (j+1) circle (.1) node[below, red]{$\mu_{j\!+\!1}$};

\node at (1.5, 1) {$C$};
	
	\end{tikzpicture}	
}

\caption{Illustration of the structure for the multispecies density optimal control problem. Red dots correspond to marginal constraints \eqref{eq:multispecies_discrete_rewritten2}, blue lines to bi-marginal constraints \eqref{eq:multispecies_discrete_rewritten3} and \eqref{eq:multispecies_discrete_rewritten4}, and yellow triangles to cost interactions. Note that marginal $\mu_0$ does not have a constraint, thus for the corresponding triangle in (b), the bottom left node is not colored red.\vspace{-0.2cm}}	
	\label{fig:multispecies_triangles}
\end{figure}

In contrast to previous works, the cost tensor $\newC$ is not composed of pairwise cost interactions, and problem \eqref{eq:multispecies_discrete_rewritten} does not fall into the framework for graph-structured optimal transport and tensor optimization problems \cite{haasler2021multimarginal, ringh2021graph, haasler2020multi}.
However, similar to the setting in Section~\ref{subsec:ot}, the solution to \eqref{eq:multispecies_discrete_rewritten} is of the form $\bM = \bK \odot \bU$, where
\begin{equation}\label{eq:K_multi_species_new}
\bK_{\ell i_{0} \dots i_\ccT} = \prod_{j = 0}^{\ccT - 1} K_{\ell i_{j} i_{j+1}},
\end{equation}
with $K_{\ell i_{j} i_{j+1}} = \exp(- C_{\ell i_{j} i_{j+1}}/\epsilon)$,
and where
\begin{equation}\label{eq:U_multi_species_new}
\bU_{\ell i_{0} \dots i_\ccT} = (U_{-1,0})_{\ell i_{0}} \! \left(\prod_{j = 1}^\ccT (U_{-1,j})_{\ell i_{j}} \! \right) \!\! \left( \prod_{j = 1}^\ccT (u_j)_{i_j} \! \right) \! .
\end{equation}
This can be readily derived using Lagrangian relaxation and is omitted for brevity (cf.~\cite{haasler2021multimarginal, ringh2021graph, haasler2020multi}).
Moreover, the components of the tensor $\bU$ can be found by generalized Sinkhorn iterations \cite{ringh2021graph}.
In particular, the problem can be solved by Algorithm~\ref{alg:multi_species}. in which $^*$ denotes the Fenchel conjugate of a function and $\partial$ denotes the subdifferential (for definitions, see, e.g., \cite{bauschke2017convex}).
Under relatively mild conditions on the cost functions, the algorithm is in fact globally convergent (see \cite[Sec.~III]{ringh2021graph} for details).
Akin to the classical Sinkhorn iterations \eqref{eq:sinkhorn_multi}, the computational bottleneck is to compute the relevant marginal and bi-marginal projections of the tensor $\bM = \bK \odot \bU$.
An efficient way to compute these is described in the following Theorem,
and the structure of these computations are illustrated in Figure~\ref{fig:multispecies_triangle_computational}.

\begin{theorem} \label{thm:projections}
The bi-marginal projections of the tensor $\bM = \bK \odot \bU$, with $\bK$ and $\bU$ as in \eqref{eq:K_multi_species_new} and \eqref{eq:U_multi_species_new}, respectively, on the marginals $-1$ and $j$,
are given by
\begin{equation} \label{eq:proj_j}
P_{-1,j}(\bK \odot \bU) =   \Psi_j \odot \hat{\Psi}_j \odot  U_{-1,j} \diag(u_j),
\end{equation}
for $j=1,\dots,\ccT-1$, and
\begin{align} 
 P_{-1,0}(\bK \odot \bU)  &=  \Psi_0\odot U_{-1,0}  \label{eq:proj_0} \\
P_{-1,\ccT} (\bK \odot \bU) &=    \hat{\Psi}_\ccT \odot U_{-1,\ccT}\diag(u_\ccT).  \label{eq:proj_T}
\end{align}
The components in these expressions are defined as
\begin{equation*}
\hat{\Psi}_j \! = \!
\begin{cases}
\! \hat{S}_K (U_{-1,0})  & j = 1\\
\! \hat{S}_K  ( \hat{\Psi}_{{j-1}}  \odot  U_{-1, j-1} \text{diag}(u_{j-1}) )  , & j=2, \ldots, \ccT
\end{cases}
\end{equation*}
with
\begin{equation}\label{eq:hatSK}
 \left( \hat S_K (A) \right)_{i_1 i_3}  = \sum_{i_2} K_{i_1 i_2 i_3} A_{i_1 i_2},
\end{equation}
and
\begin{equation*}
\Psi_j \! = \! 
\begin{cases}
\! S_K ( U_{-1, \ccT} \text{diag}(u_{\ccT})) & j = \ccT - 1 \\
\! S_K (  \Psi_{j+1} \odot  U_{-1, j+1} \text{diag}(u_{j+1}) ) , & j=0, \ldots, \ccT-2,
\end{cases}
\end{equation*}
with
\begin{equation}\label{eq:SK}
\left( S_K (A) \right)_{i_1 i_2}  = \sum_{i_3} K_{i_1 i_2 i_3} A_{i_1 i_3}.
\end{equation}
\end{theorem}

\begin{proof}
Note that
$\bK_{\ell i_0 \ldots i_\ccT} = \prod_{t = 0}^{\ccT-1} K_{\ell i_{t} i_{t+1}}$.
Together with \eqref{eq:U_multi_species_new}, this means that
\begin{align*}
& (P_{-1, j}(\bK \odot \bU))_{\ell i_j} = \sum_{\substack{i_0, \ldots, i_{j-1} \\ i_{j+1}, \ldots, i_{\ccT}}} \Bigg( \!\! \left(\prod_{t = 0}^{\ccT-1} K_{\ell i_{t} i_{t+1}} (U_{-1,0})_{\ell i_{0}} \! \right) \\
&\!\! \left(\prod_{t = 1}^\ccT (U_{-1,t})_{\ell i_{t}} \!\! \right) \! \!\! \left(\prod_{ t=1}^\ccT (u_t)_{i_t} \!\! \right) \!\!\! \Bigg) \!\! = (U_{-1,j})_{\ell i_{j}}(u_j)_{i_j} (\hat{\Psi}_j)_{\ell i_j} (\Psi_j)_{\ell i_j},
\end{align*}
where
$\hat \Psi_j$ and $\Psi_j$, for $j=1,\dots,\ccT$, are given by
\begin{align*}
\!\! (\hat{\Psi}_j)_{\ell i_j} \! & = \!\!\!\!\! \sum_{i_0, \ldots, i_{j-1} } \!\!\!\!  K_{\ell i_0 i_1} (U_{-1,0})_{\ell i_{0}} \! \left( \prod_{t = 1}^{j-1} K_{\ell i_{t} i_{t+1}} (U_{-1,t})_{\ell i_{t}} (u_t)_{i_t} \! \right) \!\! , \\
\!\!  (\Psi_j)_{\ell i_j} \! & = \sum_{i_{j+1}, \ldots, i_{\ccT}} \left(\prod_{t = j+1}^{\ccT} K_{\ell i_{t-1} i_{t}} (U_{-1,t})_{\ell i_{t}} (u_t)_{i_t} \right) \!\! . 
\end{align*}
This proves \eqref{eq:proj_j}, and similar derivations (omitted due to space constraints) yield the expressions \eqref{eq:proj_0} and \eqref{eq:proj_T}.
\end{proof}

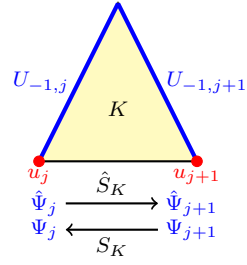
\begin{SCfigure}[1.2]
\begin{tikzpicture}[scale=0.7]
	\footnotesize
	\tikzstyle{main}=[circle, fill=white, minimum size = 12mm, thick, draw =black!80, node distance = 10mm]
	
\coordinate (-1) at (1.5, 3) ;
\coordinate (j) at (0,0);
\coordinate (j+1) at (3,0);	
	
\draw[thick,fill=yellow!30] (-1) -- (j)  -- (j+1) -- cycle;		
\draw[blue, ultra thick] (-1) --node[left]{$U_{-1,j}$} (j);
\draw[blue, ultra thick] (-1) --node[right]{$U_{-1,j+1}$} (j+1);

\draw[red, fill=red]   (j) circle (.1) node[below, red]{$u_{j}$}; 
\draw[red, fill=red]   (j+1) circle (.1) node[below, red]{$u_{j+1}$};

\node at (1.5, 1) {$K$};

\node[blue] (psij) at (0.1,-0.8) {$ \hat{\Psi}_j$};
\node[blue] (psij+1) at (2.9,-0.8) {$ \hat{\Psi}_{j+1}$};

\draw [->, thick] (psij) -- node[above]{$\hat{S}_K$} (psij+1);

\node[blue] (psihj) at (0.1,-1.3) {$\Psi_j$};
\node[blue] (psihj+1) at (2.9,-1.3) {$\Psi_{j+1}$};		

\draw [<-, thick] (psihj) -- node[below]{$S_K$} (psihj+1);
	
	\end{tikzpicture}
	\caption{ Illustration of how the structure of problem \eqref{eq:multispecies_discrete_rewritten} can be used for computations. The components in each triangle in Figure~\ref{fig:multispecies_triangles} give rise to computational components as indicated in this Figure. The operators $\hat{S}_K$ and $S_K$, defined in \eqref{eq:hatSK} and \eqref{eq:SK}, respectively, map these components forward and backwards over the triangles.} 	
	\label{fig:multispecies_triangle_computational}
\end{SCfigure}

\begin{remark}
The expressions in Theorem~\ref{thm:projections} can be seen as a message-passing scheme similar to \cite{haasler2020multi,fan2022complexity}.
More precisely, the operators $\hat S_K$ and $S_K$ are then interpreted as messages that propagate information forward and backwards, respectively, through the time instances $t=0,\dots,\ccT$.
Moreover, it is easy to adapt this to accommodate time-varying dynamics. In this case,
$\bC_{\ell i_0 \ldots i_\ccT} = \sum_{j = 0}^{\ccT - 1} C^j_{\ell i_j i_{j+1}}$ and $\bK_{\ell i_{0} \dots i_\ccT} = \prod_{j = 0}^{\ccT - 1} K^j_{\ell i_{j} i_{j+1}}$,
and $\hat S_K$ and $S_K$ in \eqref{eq:hatSK} and \eqref{eq:SK} are changed to $\hat S_{K^{j-1}}$ and $S_{K^j}$, respectively (cf.~Figure~\ref{fig:multispecies_triangle_computational}).
\end{remark}

\begin{algorithm}[tb]
	\begin{algorithmic}[1]
              \STATE Given: Initial guess $u_1,\dots, u_{\ccT}$, $U_{-1,0}, \ldots, U_{-1,\ccT}$.
              \STATE Denote $\mathscr{G}^{L}$ by $\mathscr{F}_\ccT^{L}$, and $G$ by $F_{\ccT}$.
		\WHILE{Not converged}	
		\FOR{ $j=1,\dots,\ccT$}
		\STATE Let $W_{-1, j}$ be so that $P_{-1,j}(\bK \odot \bU) \!=\!  U_{-1,j} \odot W_{-1, j}$. 
		Update $U_{-1,j}$ by solving \\$0 \in - U_{-1, j} \odot W_{-1, j} + \partial((\mathscr{F}_j^{L})^*)\big(-\! \epsilon\log(U_{-1, j}) \big)$.
		\STATE Let $w_{j}$ be so that $P_{j}(\bK \odot \bU) = u_j \odot w_j$. Update $u_j$ by solving $0 \in - u_{j} \odot w_{j} + \partial (F_{j}^*)\big(- \epsilon \log( u_{j})\big).$
		\ENDFOR
		\ENDWHILE
	  \RETURN $u_1,\dots,u_{\ccT}$, $U_{-1,0}, \ldots, U_{-1,\ccT}$
	\end{algorithmic}
	\caption{Method for solving \eqref{eq:multispecies_discrete_rewritten}.}
	\label{alg:multi_species}
\end{algorithm}

\section{Numerical example in coordination of multiple types of robots}\label{sec:example}

In this section, we illustrate the method by considering a numerical example of a robot coordination task.
The scenario is a search-and-rescue-type mission, with $L = 3$ different types of robots and three different types of terrains. The goal for the robots is to, at the last time point, cover the entire area, and to do so as cheap as possible. The exact costs are defined below. The set-up is shown in Figure~\ref{subfig:terrain}, where
blue area is water, red area is rough terrain, and green area is normal terrain. Moreover, the three different types of robots start in the three areas marked in the lower left corner of the figure: robot type 1, which start in the dark blue starting area, can move on water and in normal terrain; robot type 2, which start in the dark red starting area, can move in rough terrain and normal terrain; and robot type 3, which start in the black starting area, can only move in normal terrain.

The state space is the rectangle $[-1,1] \times [-1,1]$, which we uniformly discretize it into $100 \times 100$ grid points; the latter are denoted $x_i = x_{i_1,i_2}$ for $i_1,i_2 = 1, \ldots, 100$, and the distance (in each direction) between discrete points is denoted $\Delta x$. Moreover, time is discretized into $\ccT+1 = 61$ time steps. The dynamics for each robot type is taken to be
$f_\ell(x) \equiv 0$ and $B_\ell(x) = (1/\sqrt{\alpha_\ell}) I$, where $1/\sqrt{\alpha_\ell}$ is a robot-type-dependent weight modeling the energy efficiency of the robot type. By the reparametrization
$\tilde{v}_\ell = (1/\sqrt{\alpha_\ell}) v_\ell$, $\alpha_\ell$ can equivalently understood as a cost of movement for robot type $\ell$.
However, the distance each type of robots can travel with one time step is also limited: robot type 1 and 2 can travel to points inside a circle of radius $\sqrt{6} \Delta x$, and robot type 3 can travel to points inside a circle of radius $3 \Delta x$. In free terrain, this results in the corresponding discrete movement stencils shown in Figure~\ref{subfig:movement}, but we also disallow robots to ``jump over'' areas where they cannot enter.
This means that the cost tensor has elements \eqref{eq:C_elem} given by
$C_{\ell i k} = \alpha_\ell \|x_{i_1, i_2} - x_{k_1, k_2}\|^2$ if, for robot type $\ell$, state $x_{k_1, k_2}$ is in range from state $x_{i_1, i_2}$, and $C_{\ell i k} = \infty$ else.
This means that the corresponding $\bK$ in \eqref{eq:K_multi_species_new} is a sparse tensor, since $K_{\ell ik} = 0$ if $C_{\ell ik} = \infty$. We set $\alpha_1 = 400$, $\alpha_2 = 400$, and $\alpha_3 = 100$.

More precisely, we consider the discrete problem
  \begin{align*}  
    \minwrt[\substack{\bM \in \RR_+^{3(100^2)^{61}}\\ \mu_j,\, \mu_j^{(\ell)} \in \RR_+^{100^2}}]
        &  \; \langle \bC, \bM \rangle \! + \! 0.2 \, D(\bM) \! + \! \sum_{j=1}^{59} \! \left( \!  F(\mu_j) \! + \! \sum_{\ell = 1}^3 \langle c_{\ell}, \mu_j^{(\ell)} \rangle \! \right) \\
    \text{subject to }  \; & \; P_{-1,j}(\bM) =[ \mu_j^{(1)}, \mu_j^{(2)}, \mu_j^{(3)}]^T,  \\
     & \;  \sum_{\ell = 1}^3\mu_j^{(\ell)} = \mu_j,
     \quad \mu_j^{(\ell)} \leq \kappa^{(\ell)}, \quad \mu_0^{(\ell)} = \mu_{0, \ell}, \\
     & \;\mbox{for } \ell = 1, 2, 3, \mbox{ and } j= 0,1,\ldots, 60, \\
     & \; \mu_{60} \mbox{ uniform outside starting areas, see below.}
  \end{align*}
The total mass of each robot type
is set to $10$,
and the starting distributions $\mu_{0, \ell}$ are set to even distribution in each robots starting area. 
As final distribution $\mu_{60}$, we enforce a uniform distribution of total mass 10 in all areas outside the starting areas; inside the starting areas, no constraint is enforced at the last time point.
Moreover, the running cost $F$ is a cost for congestion in states outside of the starting areas: $F(\mu) = \sum_{(i_1, i_2) \not \in \text{starting area}} f((\mu)_{i_1, i_2})$ where $f : \RR_+ \to \RR_+ \cup \{ \infty \}$, $f(x) = x/(1 - x) + I_{[0, 1]}(x)$ where $I_{A}(x)$ is the indicator function on a set $A$, i.e., $I_{A}(x) = 0$ if $x \in A$ and $\infty$ else.
The running cost $c_\ell$ is a fixed cost for each time step a robot is deployed, i.e., it is $0$ in the starting area of robot type $\ell$ and equal to a constant $\tilde{c}_\ell$ in all other points in state space. In particular, $\tilde{c}_1 = 0.2$, $\tilde{c}_2 = 0.2$, and $\tilde{c}_3 = 0.1$.
The constraint $\kappa^{(\ell)}$ is zero in regions where the robot type cannot move (including other robot types starting areas), it is 10 in the starting regions of robot type $\ell$, and it is 1 elsewhere. The latter has no impact on the optimal solution, since the running cost $F$ limits the total density to $1$ in any given point.

The problem is solved using Algorithm~\ref{alg:multi_species}, where the projections needed in the algorithm are computed using Theorem~\ref{thm:projections}. The optimal solution is shown in Figure~\ref{fig:ex_solution}.
It is nontrivial to allocate the robots due to interaction costs and capacity constraints, nevertheless, the behaviors in the solution is consistent with the intuition. Robot type 1 and type 2 focus on the regions where only they can reach, while robot type 3 covers most of the area that all robots can reach due to the smaller cost and larger speed. It can also be seen that deployment of the robots is delayed as much as possible due to the cost of being outside the starting region.

\begin{figure}[tbh]
\begin{center}
  \subfigure[Terrain and starting positions.]{
      \includegraphics[trim=4.5cm 1.2cm 4.5cm 1.1cm, clip=true, width=0.27\textwidth]{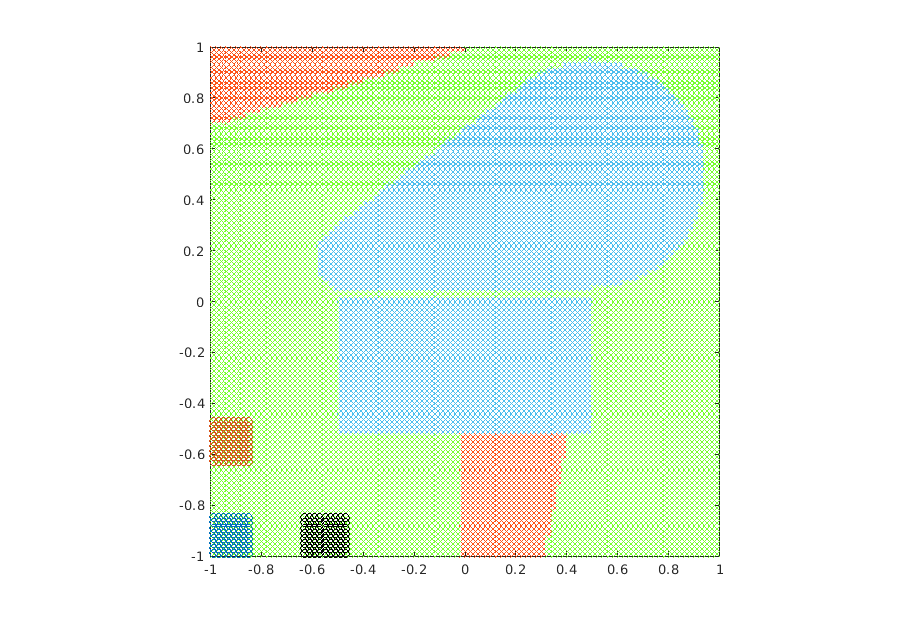}
      \label{subfig:terrain} }
      \subfigure[Movement patterns.]{
      \includegraphics[trim=4cm 3cm 10cm 4cm, clip=true,width=0.17\textwidth]{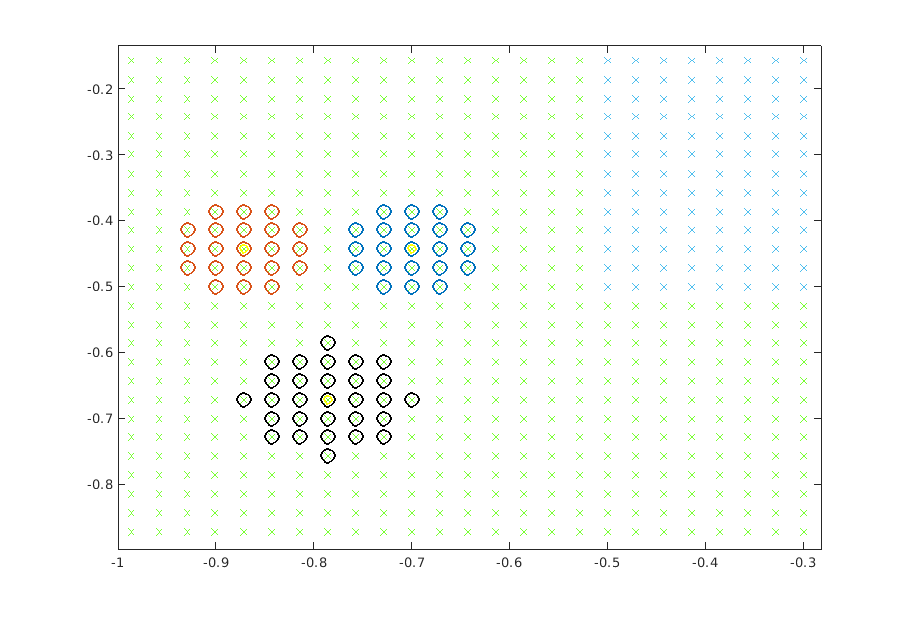}
      \label{subfig:movement}
      }
     \caption{Set-up for the numerical example. In (a), blue area is water, red area is rough terrain, and green area is normal terrain. The three different types of robots start in the three areas marked in the lower left corner: robot type 1 in the dark blue area, robot type 2 in the dark red area, and robot type 3 in the black area. In (b), the color of the (free) movement stencils are the same as the starting positions of the corresponding robot type. \vspace{-0.5cm}}
    \label{fig:ex_setup}
  \end{center}
\end{figure}

\begin{figure*}[tbh]
  \centering
  \includegraphics[trim=1.5cm 0.5cm 2.2cm 0.2cm, clip=true, width=0.98\textwidth]{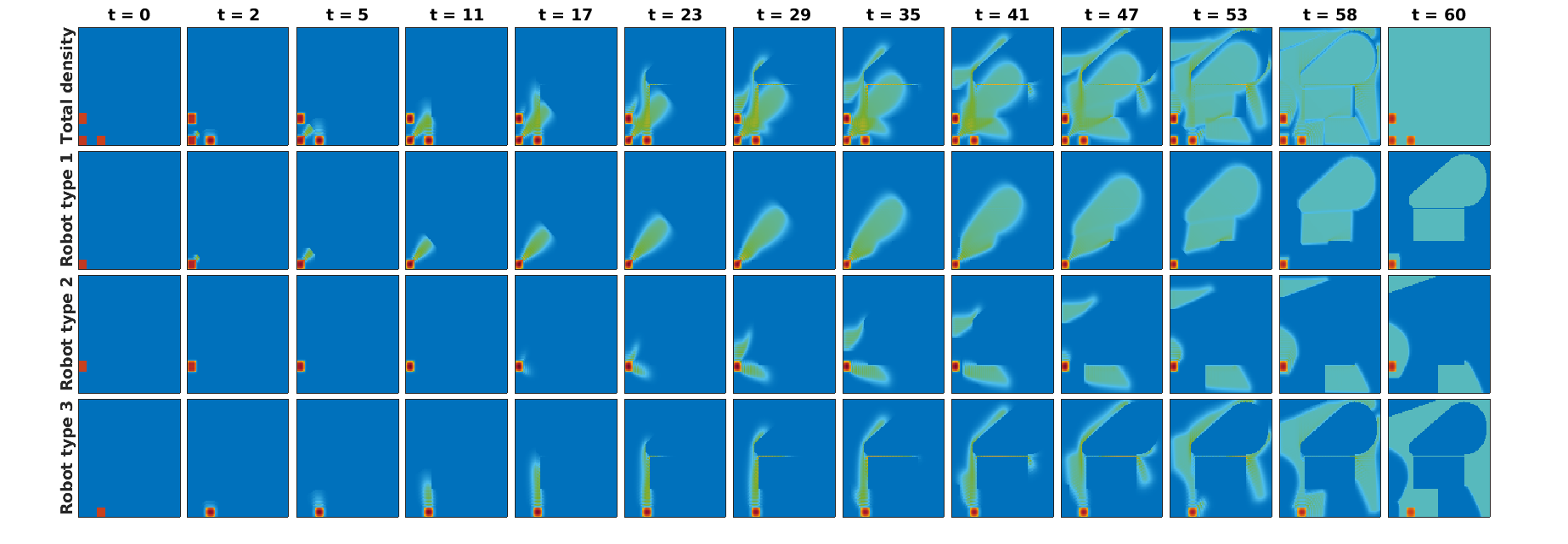}
  \caption{Time evolution of optimal total density and optimal densities of the different robots types.}
  \label{fig:ex_solution}
\end{figure*}

\section{Conclusions and future directions}\label{sec:conclusions}
In this work we developed an efficient method for multispecies mean field type control problems, where each species have different dynamics.
We also illustrated its use by solving a robot coordination task for a search-and-rescue-type scenario. One limitation of our method is that it becomes ill-conditioned when the intensity of noise $\epsilon$ becomes too small. One future direction is to address this issue by incorporating ideas from proximal point methods.

\balance
\bibliographystyle{plain}
\bibliography{./ref}

\end{document}